\documentclass[a4paper,10pt,leqno,english]{smfart}
\usepackage{aeguill}
\usepackage{enumerate}
\usepackage{amssymb,amsmath,latexsym,amsthm}
\usepackage[T1]{fontenc}
\usepackage{smfthm}      
\usepackage{geometry}   
\usepackage{url}      
\usepackage[frenchb, english]{babel}
\usepackage[utf8]{inputenc}   
\usepackage{mathrsfs}
\usepackage{relsize}
\usepackage{xcolor}
\usepackage{comment}
\definecolor{violet}{rgb}{0.0,0.2,0.7}
\definecolor{rouge2}{rgb}{0.8,0.0,0.2}
\usepackage{hyperref}
\usepackage{enumitem}
\setlist[enumerate]{leftmargin=*}
\usepackage{mathrsfs}
\let\originallhook\lhook
\usepackage{MnSymbol}
\let\lhook\originallhook
\usepackage[all]{xy}
\usepackage[fleqn,tbtags]{mathtools}
\hypersetup{
    bookmarks=true,         
    unicode=false,          
    pdftoolbar=true,        
    pdfmenubar=true,        
    pdffitwindow=false,     
    pdfstartview={FitH},    
    pdftitle={},    
    pdfauthor={},     
    colorlinks=true,       
   linkcolor=violet,          
    citecolor=violet,        
    filecolor=black,      
    urlcolor=cyan}           
\renewcommand{\P}{\mathbb{P}}
\newcommand{\R}{\mathbb{R}}
\newcommand{\CC}{\mathbb{C}}

\newcommand{\D}{\mathbb D}

\newcommand{\vp}{\varphi}

\renewcommand{\O}{\mathcal{O}}

\renewcommand{\le}{\leqslant}

\newcommand{\Ric}{\mathrm{Ric} \,}
\newcommand{\om}{\omega}
\newcommand{\omvp}{\omega_{\varphi}}
\newcommand{\omb}{\omega_{\beta}}

\newcommand{\xreg}{X_{\rm reg}}

\renewcommand{\D}{\mathbb D}

\newcommand{\bom}{\bar \om}

\newcommand{\Supp}{\mathrm {Supp}}
\newcommand{\ddc}{dd^c}
\newcommand{\Sp}{\mathrm{Spec}}
\newcommand{\xdr}{(X,D)_{\rm reg}}
 \newcommand{\longhookrightarrow}{\ensuremath{\lhook\joinrel\relbar\joinrel\rightarrow}}

\newtheorem*{mthm}{Main Theorem}

\numberwithin{equation}{section}

\begin{document}

\frontmatter 

\title[KE metrics with conic singularities along self-intersecting divisors]{Kähler-Einstein metrics with conic singularities along self-intersecting divisors}

\date{\today}
\author{Henri Guenancia}
\address{Department of Mathematics \\
Stony Brook University, Stony Brook, NY 11794-3651}
\email{guenancia@math.sunysb.edu}
\urladdr{www.math.sunysb.edu/{~}guenancia}


\begin{abstract}
In this paper, we extend the existence and regularity theorems for Kähler-Einstein metrics having conic singularities along a simple normal crossing divisor to the case of normal crossing divisor, i.e. when components of the divisor are allowed to intersect themselves transversely.
\end{abstract}

\maketitle
\mainmatter

\section*{Introduction}

Let $X$ be a complex Kähler manifold. A divisor $D$ (formal sum of hypersurfaces) is said to have normal crossing support if near any point in its support, $\Supp(D)$ is given by $(z_1 \,\cdots \, z_d=0)$ for some holomorphic system of coordinates $(z_i)$. 

With a given $\R$-divisor $D= \sum (1-\beta_k)D_k$ with normal crossing support such that $\beta_k\in (0,1)$ for all $k$, we can associate the notion of Kähler metric with \textit{conic singularities} along $D$. For our purposes in this paper, such 
an object is a Kähler metric $\om$ on $X\setminus(\cup D_k)$ which is quasi-isometric to the model metric with conic singularities: more precisely, near each point $p\in \Supp(D)$ where $\Supp(D)$ is defined by the equation $(z_1 \, \cdots \, z_d=0)$, we ask $\om$ to satisfy:
\[C^{-1} \om_{\rm cone} \le \om \le C \om_{\rm cone}\]
for some constant $C>0$, and where 
\[\om_{\rm cone}:=\sum_{k=1}^d \frac{1}{|z_k|^{2(1-\beta_k)}}i dz_k\wedge d\bar z_k +\sum_{k=d+1}^n i dz_k\wedge d\bar z_k \] 
is the model cone metric with cone angles $2\pi \beta_k$ along $(z_k=0)$.

Given a pair $(X,D)$ as above, a natural question to ask is whether one can find a \textit{Kähler-Einstein} metric $\om$ on $X\setminus \Supp(D)$ (i.e. satisfying $\Ric \om = \mu \om$ for some $\mu \in \R$  on this open subset) having conic singularities along $D$. Such a metric will be referred to as a \textit{conic Kähler-Einstein metric}.\\

This question has been studied a lot recently, and gave rise to a number of works such as \cite{Don}, \cite{Brendle}, \cite{CGP}, \cite{JMR} or \cite{GP}. In these papers, the involved divisor (or, more precisely, its support) was always assumed to be smooth or merely a simple normal crossing divisor. This last notion is a convenient and current strengthening of the notion of normal crossing divisor where one also requires the components of the divisor to be smooth (typically one excludes self-intersecting divisors, even if the crossing is transverse). This notion is much widespread in algebraic geometry, particularly because of Hironaka's desingularization theorem, which produces such objects out of arbitrary singular ones. 

In any case, it is surprising that constructing Kähler-Einstein metrics with conic singularity along a given divisor (a notion that is local in the analytic topology) has so far always required working with simple normal crossing divisors (a notion which is local only in the Zariski topology). The reason for that is the difficulty to construct global Kähler metrics having conic singularities along a self-intersecting divisor, as we point it out in more details in \S \ref{sec:prev}. \\

In this paper, we investigate the question when $X$ is a projective manifold, and more precisely we prove:

\begin{mthm}
Let $X$ be a complex projective manifold of dimension $n$, $D= \sum (1-\beta_k)D_k$ a divisor with normal crossing support such that $\beta_k\in (0,1)$ for all $k$, and let $\mu \in \R$. Let $\om$ be any closed positive $(1,1)$ current with bounded potentials satisfying:
$$\Ric \om = \mu \om + [D]$$
Then $\om$ has conic singularities along $D$, i.e. $\om$ is a Kähler-Einstein conic metric. 
\end{mthm}

So this theorem is a regularity theorem: starting with a current whose potentials are bounded but satisfies a Kähler-Einstein equation, we deduce that the current is actually smooth outside $D$ and is quasi-isometric to the model cone metric near the divisor. As for the question of existence, it has been essentially settled in the work of Ko\l odziej \cite{Kolo} (at least when $\mu$ is nonpositive) and in \cite{rber, BBEGZ} for the case where $\mu$ is positive, given suitable properness assumptions.\\

Let us try now to outline the new difficulty compared to the snc case. In order to reduce the problem to the snc case, one would be tempted to blow up the self-intersection locus of the divisor. This will indeed make the divisor into an snc one, but the pulled-back metric will then live in a non Kähler class, and this is still a major and unsolved problem to get regularity properties for degenerate Monge-Ampère equations near the non-Kähler locus of the background cohomology class. So instead of performing a blow-up, which would kill ampleness, we want to resolve the singularities using a \textit{finite} morphism.  This turns out to always be possible locally, but when one wants to do it globally on $X$, one necessarily puts some ramification and singularities into the game (i.e. the finite morphism $f:Y\to X$ that we are looking for will be ramified and involve a \textit{singular} variety $Y$). 

Still we are able to address these issues, relying on the main result of \cite{G2} (generalized in \cite{GP} to arbitrary angles $\beta\in (0,1)$) asserting that a weak Kähler-Einstein metric for a klt pair has conic singularities on the simple normal crossing locus of the pair. We will proceed in four steps:

\begin{enumerate}[label= \textit{Step \arabic*.}]
\item Using the theorem above, reduce the question to a local one, near points $x\in X$ where components of the divisors intersect themselves.
\item Find an étale neighborhood $U\to X$ of $x$ where $D$ has simple normal crossing support (this uses Artin's approximation theorem)
\item Compactify $U$ to get a \textit{finite} map $f:Y\to X$ extending $U\to X$.
\item Identify $f^*\om$ to the Kähler-Einstein metric of a klt pair $(Y, \Delta)$ whose snc locus contains $U$ and apply \cite{G2} to conclude.\\
\end{enumerate}

Finally, in the last section \S \ref{sec:nodal} we make explicit the construction of the above map $f:Y\to X$ when $X=\P^2$ and $D=(y^2=x^2(x+1))$ is the nodal cubic. It turns out that we can choose $Y=\P^2$ and $f:\P^2 \to \P^2$, a degree 9 morphism that extends the normalization $\P^1 \to D$.\\

\noindent
\textbf{Acknowledgments. } I am grateful to Chengjian Yao who brought this question to my attention and with whom I had very productive discussions. I would also like to warmly thank Olivier Benoist for his very insightful and helpful suggestions.

\section{Normal crossing vs simple normal crossing divisors}

Let us start with an example, or better, two examples. 
The first example is given by $\CC^2$ endowed with the hypersurface $H=(xy=0)$, where $x,y$ are the standard holomorphic coordinates on $\CC^2$. As a divisor, $H$ can be decomposed as the sum $(x=0)+(y=0)$ and is the typical example of what is called a simple normal crossing (snc for short) divisor. 

Now, still in $\CC^2$, consider the nodal cubic $C=(y^2=x^3+x^2)$. The curve $C$ is \textit{irreducible} and has a singularity at the origin (a node) that looks similar to the singularity of $H$ at the origin, at least in the analytic topology. For this reason, we call $C$ a normal crossing (nc for short) divisor on $\CC^2$. 

The general definitions are given below:

\begin{defi}
Let $X$ be a smooth complex manifold of dimension $n$ and $D$ a reduced divisor on $X$. We say that $D$ has

$\bullet$ \textit{normal crossings} if $D$ is locally analytically given by the union of coordinate hyperplane, i.e. if for every $p\in D$, there exists a analytic chart $U \ni p$ and coordinates $z_1, \ldots, z_n$ on $U$ such that $D\cap U = (z_1 \cdots z_r=0)$ for some $1 \le r \le n$.

$\bullet$ \textit{simple normal crossings} if $D$ has normal crossings and every irreducible component of $D$ is smooth.

\end{defi}

These two notions are very close with each other, however, having normal crossings is local in the \textit{analytic} topology while having simple normal crossings is local only in the \textit{Zariski} topology, as shown by the example above. However, a nc divisor is a snc divisor in the étale topology:

\begin{prop}
\label{prop:etale}
Given a normal crossing divisor $D$ on $X$ and a point $p\in D$, there exists an étale map $f:U\to X$ such that $p\in f(U)$ and $f^*D$ is a simple normal crossing divisor.
\end{prop} 

Recall that an étale map is a flat and unramified map (i.e. every schematic fiber is finite and reduced), which in our situation is the same as being a local biholomorphism. It is useful to remember that locally all étale morphisms are induced by maps of the form $A \to A[T_1, \ldots, T_N]/(f_1, \ldots, f_N)$ where $\mathrm{Jac}(f_1, \ldots, f_N)$ is a unit of $A[T_1, \ldots, T_N]/(f_1, \ldots, f_N)$.\\

Proposition \ref{prop:etale} is well-known to the experts and is an easy consequence of Artin's approximation theorem \cite[Theorem 1.10]{Artin}, but we are going to give a proof for the convenience of the reader. We first need to recall a few notions about étale local rings and Henselian rings; we refer to \cite{Milne} for further details and proofs.\\

Right now, $X$ is a complex variety (possibly singular), and $x\in X$ denotes a (closed) point. The local ring at $x$ for the étale topology is defined to be
\[\mathcal O_{X, \bar x}=\varinjlim _{(U,u)}\mathcal O_{U,u}\]
where $(U,u)$ varies over all étale maps $U \to X$ sending $u$ to $x$, $U$ being connected and affine. This is a noetherian local ring (whose maximal ideal is given by $\mathfrak m_{X,\bar x}:= \varinjlim \mathfrak m_{U,u}$) which happens to be Henselian, i.e. given $f_1, \ldots, f_r \in \mathcal O_{X, \bar x}[T_1, \ldots, T_r]$, then every common zero $\mathbf{a}_0$ in $\CC^n$ of the $\bar f_i$ (which denotes the image of $f_i$ under the evaluation map induced by $\mathcal O_{X, \bar x} \to \mathcal O_{X, \bar x}/\mathfrak m_{X,\bar x} \simeq \CC$) for which $\mathrm{Jac}(f_1, \ldots, f_r)(\mathbf{a}_0)$ is nonzero lifts to a common zero of the $f_i$ in $\mathcal O_{X, \bar x}^n$. Actually, we can even say much more, as $\mathcal O_{X, \bar x}$ is the Henselianization of $\mathcal O_{X, x}$, that is every local homomorphism $\mathcal O_{X, x} \to B$ with $B$ Henselian factors uniquely into $\mathcal O_{X, x} \to \mathcal O_{X, \bar x} \to B$ where the first map is the canonical one provided by the observation that any Zariski neighborhood is an étale neighborhood. 

Suppose now that $X$ is regular at $x$. Then one can show that $\mathcal O_{X, \bar x}$ only depends on $\mathrm{dim}(X)$, and is isomorphic to $\mathcal O_{\CC^n, \bar 0}$, which equals:
\[\mathcal O_{\CC^n, \bar 0}=\CC [[T_1, \ldots, T_n]] \cap \CC(T_1, \ldots, T_n)^{\rm al}\]
i.e. it consists of the formal series in $n$ variables that are roots of a polynomial with coefficients in $\CC [T_1, \ldots, T_n]$. So it is strictly smaller than the formal local ring of $\CC^n $ at $0$ (which consists of all formal series), but the following particular case of a general approximation theorem due to M. Artin \cite[Theorem 1.10]{Artin} shows the close relationship between these two rings: 

\begin{theo}[Artin's Approximation theorem]
Let $X$ be a complex variety, and $x\in X$ a closed point. Given an arbitrary system of polynomial equations 
$$f(\mathbf{Y})=0, \quad \mathbf Y=(Y_1, \ldots, Y_N)$$
with coefficients in $\mathcal O_{X, \bar x}$, a solution $\hat y = (\hat y_1, \ldots, \hat y_N)$ in the $\mathfrak m_{X,x}$-adic completion $\widehat{\mathcal O}_{X, x}$ and an integer $k$, there exists a solution $y=(y_1, \ldots, y_N)$ in $\mathcal O_{X, \bar x}$ with
$$y_i \equiv \hat y_i  \quad (\mathrm{modulo} \, \,  \mathfrak m_{X,x}^k)$$
\end{theo}

Let us apply this theorem to the very concrete situation given by Proposition \ref{prop:etale}:

\begin{proof}[Proof of Proposition \ref{prop:etale}] We start with a smooth complex variety $X$ and a normal crossing divisor $D$ on $X$. Pick a point $x\in D$; there is a (Zariski) neighborhood of $x$ where $D$ is given by $f=0$ where $f \in \mathcal O_{X,x} \simeq \CC[T_1, \ldots, T_n]_{(0)} \subset \CC(T_1, \ldots, T_n)$. As $D$ has normal crossings, there exists a regular sequence $\hat g_1, \ldots, \hat g_r \in \widehat{\mathcal O}_{X, x} \simeq \CC[[T_1, \ldots, T_n]]$ (these series actually have a positive (poly)radius of convergence) such that $ f= \hat g_1 \cdots \hat g_r$ in  $\widehat{\mathcal O}_{X, x}$.   Therefore the polynomial $Y_1 \cdots Y_r-f \in \mathcal O_{X, x}[Y_1, \ldots, Y_r]$ has a solution $\hat g = (\hat g_1, \ldots, \hat g_r)$ in $\widehat{\mathcal O}_{X, x}$. Using Artin's approximation theorem with $k=2$, we find $g=(g_1, \ldots, g_r)$ in $\mathcal O_{X, \bar x}$ such that 
$g_i \equiv \hat g_i \, (\mathrm{modulo} \, \,  \mathfrak m_{X,x}^2)$. In particular $\mathrm{Jac}(g_1, \ldots, g_r)$ does not vanish at $x$ (or at $0$ via the standard identifications), so that $(g_1, \ldots, g_r)$ is a regular sequence in $\mathcal O_{X, \bar x}$. By the definition of this ring, there exists a morphism $\vp:U \to X$ étale over $x$ sending $u\in U$ to $x$ such that $g_i \in \mathcal O_{U,u}$ for every $i$. In particular, $\vp^*D$ is given by $(g_1 \cdots g_r =0)$ near $u$, and as the $g_i$'s form a regular sequence, $\vp^*D$ is snc near $u$. Up to shrinking $U$ (around $u$) one can assume that the functions $g_i$ are defined on the whole $U$ and that $U\to X$ is étale, which concludes the proof. 
\end{proof}

Informally, we started with a decomposition $f=\hat g_1, \ldots, \hat g_r $ where the $\hat g_i$'s are formal series forming a regular sequence; thanks to Artin's theorem, we can approximate these formal series using series that are roots of a polynomial. Then, up to passing to an étale cover, one can view these series as rational fractions (whose denominator does not belong to $(T_1, \ldots, T_n)$) that still form a regular sequence. \\

Let us give an example to illustrate this. We consider $D=(y^2=x^2(1+x)) \subset \mathbb A^2$ the plane nodal cubic, singular at the origin. We are looking for a square root of $1+x$ in a étale neighborhood of $0$ \--- obviously such a function does not exist in any Zariski neighborhood of $0$. We consider the morphism $\mathcal O_{\mathbb A^2, 0} \to \mathcal O_{\mathbb A^2, 0}[T]/(T^2-(1+x))$. As $1+x\in  \mathcal O_{\mathbb A^2, 0}^{\times}$, $2T$ is invertible in  $\mathcal O_{\mathbb A^2, 0}[T]/(T^2-(1+x))$, its inverse being $\frac 1 2 (1+x)^{-1}T$. 

Therefore the morphism $\Sp \left(  \CC[x,y,T]/(T^2-(1+x))\right) \to \mathbb A^2$ is étale over $0$ (so upstairs, near the points $(x,y,T-1)$ and $(x,y,T+1)$). 
Moreover, the divisor can be written upstairs as $(y-xT)(y+xT)=0$ which has simple normal crossings in a Zariski neighborhood of the two points (indeed, it has two components, and it has normal crossings near these points as the map is a local biholomorphism there). 

\section{Metric with conic singularities along divisors}

\subsection{Notion of conic singularities}

Let us start with the local case. In $\CC^n$, we consider the hypersurfaces $(z_k=0)$ for $1 \le k \le d$ where $d$ is an integer less than or equal to $n$. With each hypersurface $(z_k=0)$, we associate a real number $\beta_k \in (0,1)$; a way to bring together all this information is given by the convenient formalism of $\R$-divisors. Indeed, setting $D:=\sum_{k=1}^d (1-\beta_j) D_k$ where $D_k:=(z_k=0)$, all the data is encoded in $D$ \--- we will soon explain why we chose the coefficients to be $1-\beta_k$ instead of $\beta_k$. The model cone metric associated with this configuration is 
\[\om_{\rm cone}:=\sum_{k=1}^d \frac{1}{|z_k|^{2(1-\beta_k)}}\sqrt{-1}dz_k\wedge d\bar z_k +\sum_{k=d+1}^n \sqrt{-1}dz_k\wedge d\bar z_k \] 
One can actually interpret $\om_{\rm cone}$ as the quotient metric induced by the euclidian metric after gluing together the edges of cones of angle $2\pi \beta_k$. 

In general, given a Kähler manifold $X$ and a divisor $D=\sum (1-\beta_i) D_i$ with coefficients $\beta_i \in (0,1)$ and normal crossing support, one says that a (Kähler) metric $\om$ on the complement $X\smallsetminus \Supp(D)$ has conic singularities along $D$ if for any point $p\in \Supp(D)$, there exists an euclidian neighborhood $U$ of $p$ such that $(U,D_{|U})$ is isomorphic to $(\D^n, \sum_{k=1}^d (1-\beta_j) D_k$ )  ($\D^n$ being the unit polydisk of $\CC^n$) such that $\om$ satisfies
\[C^{-1} \om_{\rm cone} \le \om \le C \om_{\rm cone}\]
on $U$, for some constant $C>0$. \\

Kähler metrics with conic singularities along a divisors have been studied intensively, in particular lately in the context of Kähler-Einstein metrics, i.e. metrics with constant Ricci curvature.

Let us expand a little more on that topic, and choose a pair $(X,D)$ consisting of a compact Kähler manifold $X$ and a divisor $D=\sum (1-\beta_k) D_k$ as above with normal crossing support and coefficients in $(0,1)$. A Kähler-Einstein metric for this pair is a Kähler metric on $X \smallsetminus \Supp(D)$ having constant Ricci curvature and having conic singularities along $D$. It is well known that conic Kähler-Einstein metrics are related to the following type of singular Monge-Ampère equation
\[(\om+\ddc \vp)^n = \frac{e^{\mu \vp}dV}{\prod |s_k|^{2(1-\beta_k)}} \leqno{\rm (MA)}\] 
where $\om$ is a background Kähler metric on $X$, $\mu \in \R$ is a parameter which could be related to the sign of the curvature, $dV$ is some suitable smooth volume form on $X$, and $s_k$ are sections of $\mathcal O(D_k)$ defining the hypersurface $D_k$; finally, $\vp$ is a \textit{bounded} $\om$-psh function.

\noindent
If $dV$ is chosen according to the cohomological positivity properties of $K_X+ D$ (if any), then a solution $\omvp:=\om+\ddc \vp$ of (MA) satisfies \[\Ric \omvp = -\mu \omvp + [D] \leqno {\mathrm{(KE)}}\] where $\Ric \omvp := -\ddc \log \omvp^n$ (it is automatically well-defined as a current). \\

\subsection{Previous results}
\label{sec:prev}
Hence in order to construct Kähler-Einstein conic metrics a first step would be to solve the equation (MA). We remark that it is a priori not clear that a solution of (MA) will have conic singularities along 
$D$ \--- even if by the general theory the function $\varphi$ is smooth outside of the support of the divisor \---. Indeed, the equations (KE) or (MA) only impose the behavior of the \emph{determinant} of the metric $\omvp$ whereas having "conic singularities" is much more precise information about the metric itself. Nevertheless, in the case where $D$ has \textit{simple} normal crossing support, we have the following result:

\begin{theo}[\cite{GP}]
\label{thm:gp}
Let $\om$ be a Kähler metric on $X$, $dV$ a smooth volume form, $\mu \in \R$. If $D$ has simple normal crossing support, then any weak solution $\omvp=\om + \ddc \vp$ with $\vp \in L^{\infty}(X)$ of 
\[(\om+\ddc \vp)^n = \frac{e^{\mu \vp}dV}{\prod |s_k|^{2(1-\beta_k)}} \] 
has conic singularities along $D$. \\
\end{theo}
\noindent

The assumption about the divisor is crucial to produce a \textit{global} model (possibly regularized) of a Kähler metric with conic singularities along $D$. Roughly speaking, suppose that $D$ has one component, and let $s$ be a section of $D$. Then up to scaling the hermitian metric $|\cdotp|$ on $\mathcal O_X(D)$, the current $\om+ \ddc |s|^{2\beta}$ is a conic metric along $D$ as long as $D$ has simple normal crossing support because  $\om+ \ddc |z_1|^{2\beta}$ is a conic metric along $(1-\beta) [z_1=0]$ . However, if $D$ intersects itself, then locally $s$ can be written as $z_1z_2$ and it can be checked easily that $\om+ \ddc |z_1z_2|^{2\beta}$ does not have conic singularities along $(1-\beta)[z_1=0]+(1-\beta)[z_2=0)$.\\

The main goal of this paper is to show that the same result holds even when $D$ has merely normal crossing support, and we are going to rely on the singular version of Theorem \ref{thm:gp} proved in \cite[Theorem 6.2]{GP}. The following theorem deals with klt pairs as introduced in the (log) Minimal Model Program, and we refer to \cite{G2} for a detailed account of the notions involved. So, given a klt pair $(X,D)$, there is a notion of a Kähler-Einstein metric for the pair; this is a current on $X$ that turns out to be smooth and Kähler-Einstein on $\xreg \smallsetminus \Supp(D)$. In \cite{G2} and then \cite{GP}, it was proved that any such current has conic singularities on the Zariski open set $\xdr:= \{x\in X; (X,D)$ is log smooth at $x\}$:
\begin{theo}[\cite{GP}]
\label{thm:klt}
Let $(X,D)$ be a projective klt pair.
\begin{enumerate}
\item[$(i)$] If $K_X+D$ is ample, then the Kähler-Einstein metric of $(X,D)$ has cone singularities along $D$ on $\xdr$. 
\item[$(ii)$] If $K_X+D$ is numerically trivial and $\alpha$ is a Kähler class, then the Kähler-Einstein metric of $(X,D)$ living in $\alpha$ has cone singularities along $D$ on $\xdr$. 
\item[$(ii)$] If $-(K_X+D)$ is ample, then any Kähler-Einstein metric for $(X,D)$ has cone singularities along $D$ on $\xdr$.\\
\end{enumerate}
\end{theo}

\section{Proof of the Main Theorem}

Recall that we want to prove the following result:

\begin{theo}
Let $X$ be a projective manifold and $D=\sum (1-\beta_i) D_i$ a divisor with coefficients $\beta_i \in (0,1)$ and normal crossing support. Then any Kähler-Einstein metric for $(X,D)$ has conic singularities along $D$.  
\end{theo}

Recall that a Kähler-Einstein metric for $(X,D)$ is a closed positive current $\om$ living in a Kähler cohomology class which is smooth on $X\smallsetminus \Supp(D)$, had bounded local potentials, and satisfies:
\[\Ric \om = \mu \om + [D]\]
for some real number $\mu$, and where $\Ric \om := - \ddc \log \om^n$ as long as $\log \om^n \in L^{1}_{\rm loc}$.\\

\begin{proof}
We already know the result on the simple normal crossing locus of $(X,D)$, so we are going to work in near a point $x$ where the divisor does not have simple normal crossings. 

By Proposition \ref{prop:etale}, there exists a Zariski neighborhood $V$ of $x$ and an étale map $g:U\to V$ such that:

$\bullet$ $g(u)=x$ for some $u\in U$ 

$\bullet$ $g^*D$ has simple normal crossing support on $U$.\\

We consider a compactification $\overline U$ of $U$; the map $g:U \to X$ induces a rational map $ g_{\rm rat} : \overline U \dashedrightarrow X$. If we denote by $\bar X$ a desingularization of the closure of the graph of $g_{\rm rat}$, then $\overline X$ is equipped with a proper map $\bar g : \overline X \to X$ that coincides with $g$ on $U$. 
Let us now consider the Stein factorization of $\bar g$:
$$  \xymatrix{
    \overline X\ar[d]_{\bar g} \ar[r]^h  & Y \ar[ld]^f  \\
     X & 
  }$$
where $Y:= \mathbf{Spec} \,  \bar g_* \mathcal O_{\bar X} $; here $\bar g$ automatically factors through a map $h: \overline X \to Y$ satisfying $h_{*}\mathcal O_{\bar X} = \mathcal O_Y$ and we denoted by $f:Y\to X$ the canonical map coming along with $Y$. It is finite since $ \bar g_* \mathcal O_{\bar X}$ is a coherent $\mathcal O_X$-module.

\noindent
In the following lemma, we gathered all the information we need in order to complete the proof of the Main Theorem:
\begin{lemm}
\label{lem:end}
The following assertions are satisfied:
\begin{enumerate}
\item[$a.$] $h$ induces an isomorphism over $U$,
\item[$b.$] $Y$ is normal, and the pair $(Y,\Delta)$ is klt (here $\Delta:=f^*D-K_{X/Y}$),
\item[$c.$] $f^*\om$ is a Kähler-Einstein metric for $(Y,\Delta)$.\\
\end{enumerate}
\end{lemm}

\noindent
If we prove Lemma \ref{lem:end} above, we will be done with the proof of the Theorem. Indeed, by $a.$, $f^*D$ is snc on $W:=h(U)$ and $f$ is étale over $W$. As $f^*\om$ is a Kähler-Einstein metric for the klt pair $(Y, \Delta)$, Theorem \ref{thm:klt} guarantees that $f^*\om$ has conic singularities along $\Delta$ on $W$ (where $\Delta$ coincides with $f^*D$). As $f$ is étale on $W$, it is, in particular, a local biholomorphism, and therefore $\om$ has conic singularities along $D$ on $f(W)=g(U)=V$. 
\end{proof}

We are now left to prove the lemma above:

\begin{proof}[Proof of Lemma \ref{lem:end}]
Let us take things in order. 

$a.$ By Zariski's Main Theorem, $h:\overline X \to Y$ has connected fibers, and therefore so has $h_{|U}:U \to Y$ as $U$ is connected. If $h_{|U}$ is not an isomorphism onto its image, it means that there exists $y\in h(U)$ such that $U \cap h^{-1}(y)$ has positive dimension, and therefore we can find at least a curve $C \subset U$ such that $\bar g(C)=f(h(C))= pt$. This contradicts the fact that $\bar g$ is étale over $U$.    \\

$b.$ Recall that $Y$ is obtained by gluing affine schemes of the form $\Sp (H^0(\Sp \, A, \bar g_* \mathcal O_{\bar X}))$ where $\Sp \, A$ is an affine open subscheme of $X$. 
Now $H^0(\Sp \, A, \bar g_* \mathcal O_{\bar X})=H^0(\bar g^{-1}(\Sp \, A), \mathcal O_{\bar X})$ is an integrally closed ring as $\bar X$ is normal. Indeed, we claim that for any normal variety $Z$, the ring $B:=\mathcal O_Z(Z)$ is normal. First, it is obviously integral. Now, choose a monic polynomial $P$ with coefficients in $B$ and $f\in \mathrm {Frac} \, B$ satisfying $P(f)=0$. We cover $Z$ by affine varieties $\Sp \, A_i$; the restriction map induces injections $B \hookrightarrow A_i$. So we can view $f \in \mathrm{Frac} \, A_i$ for all $i$, and by normality of $A_i$, we find that $f \in \bigcap A_i$, therefore $f\in B$ and $B$ is integrally closed.  

\noindent
Another way to obtain normality of $Y$ is the following: let $\nu: \widetilde Y \to Y$ be the normalization map; by the universal property of $\nu$ (as $\overline X$ is normal) the map $h$ factors through $\nu$, i.e. $h=\nu \circ \tilde h$ for some $\tilde h:\overline X \to \widetilde Y$. The map $\tilde h $ induces an injection $\O_{\tilde Y} \hookrightarrow \tilde h_* \O_{\bar X}$; as $\nu$ is left exact, we get  $$\nu_*\O_{\tilde Y} \longhookrightarrow h_*\O_{\bar X} \overset{\sim}{\longrightarrow} \O_Y$$
that we can precompose with the natural injection $\O_{Y} \hookrightarrow \nu_*\O_{\tilde Y}$ to see that we actually have an isomorphism $\O_{Y} \simeq \nu_*\O_{\tilde Y}$. Therefore the conductor of $\nu$ is trivial and $\nu$ is an isomorphism so that $Y$ is normal.

\noindent
Now, writing $\Delta:=f^*D-K_{Y/X}$ where $K_{Y/X}:=K_Y-f^*K_X$ is the ramification divisor of $f$, we have the formula $K_Y+\Delta=f^*(K_X+D)$, and by \cite[Proposition 3.16]{Kpairs}, the pair $(Y,\Delta)$ is klt as $(X,D)$ is \--- recall that having klt singularities is local in the analytic topology, and from that perspective, $(X,D)$ is analytically log smooth even it isn't in the algebraic sense. \\

$c.$ To lighten notation, let us write $\bom := f^*\om$. This is a positive current on $Y$ with bounded potentials, so in particular its Monge-Ampère $\bom^n$ is well-defined as a non-pluripolar measure. To see that $\bom$ is a Kähler-Einstein metric for $(Y, \Delta)$, we can use the convenient characterization given in \cite[Proposition 3.8]{BBEGZ} (cf \cite[Proposition 3.3]{BG} for an analogue in the negative curvature case, the zero curvature case being very similar) that breaks down the problem to show that the following two conditions are satisfied:
\begin{enumerate}
\item[$\cdot$] $\Ric \bom = \mu \bom + [\Delta]$ on $Y_{\rm reg}$,
\item[$\cdot$] $\int_{Y_{\rm reg}} \bom^n = \{\bom\}^n$.\\
\end{enumerate}

\noindent
The second point is straightforward as $\bom$ has bounded potentials so that its volume on any Zariski open set computes the top intersection of its cohomology class. As for the first point, we first work locally near a point $y \in Y_{\rm reg}$. Using holomorphic coordinates $(z_i)$ (resp $(w_i)$) on $X$ (resp. $Y$) centered at $f(y)$ (resp. $y$), one can write $\om^n = e^h dz_1 \wedge d \bar z_1 \wedge \cdots \wedge dz_n \wedge d \bar z_n$ for some function $h\in L^{1}_{\rm loc}$. Pulling this back to $Y$ by $f$, we find:
$$\bom^n = e^{f^*h} \, |\mathrm{Jac}(f)|^2 \, dw_1 \wedge d \bar w_1 \wedge \cdots \wedge dw_n \wedge d \bar w_n $$
where $\mathrm{Jac}(f)=\Lambda^n df:f^*K_X \to K_Y$ can be viewed as a section of $K_{Y/X}$. Therefore, the identity:
$$\Ric \bom = f^* \Ric \om - [K_{Y/X}]$$
is valid on $Y_{\rm reg}$. As $\om$ is a Kähler-Einstein metric for $(X,D)$ and $\Delta=f^*D- K_{Y/X}$, we end up with:
$$\Ric \bom = \mu \bom + [\Delta]$$
on $Y_{\rm reg}$, which had to be proved. 
\end{proof}

\section{The example of the nodal cubic}
\label{sec:nodal}
Let us give a comprehensive example in the case of the nodal cubic $C=(y^2z=x^2(x+z)) \subset \P^2$. The curve $C$ is a member of 
$|\mathcal O_{\P^2}(-3)|$; for any $\beta \in (0,1)$, the pair $(\P^2,(1-\beta)C)$ is klt and we know from \cite{rber, BBEGZ} that for $\beta$ small enough, there exists a unique Kähler-Einstein metric $\omb$ satisfying $\Ric \omb = \beta \omb + (1-\beta) [C]$. We know from \cite{G2} that this metric has conic singularities outside of the node, but in order to understand what happens at the node, we would for instance need as before a finite map $\pi:S\to \P^2$ from a smooth surface $S$ to $\P^2$ such that there exists $x\in S$ with $\pi(x)=[0:0:1]$ satisfying:

$\bullet$ $\pi$ induces a biholomorphism between a neighborhood of $x$ and a neighborhood of $[0:0:1]$;

$\bullet$ As a divisor, $\pi^*C$ has simple normal crossing support near $x$.\\
Once such a morphism is found, one can argue using \cite{G2} in the same way as in the previous section. \\

We claim that the following morphism suits the requirements:
\[
\begin{array}{cccc}
\pi:& \mathbb{P}^2  & \longrightarrow & \mathbb{P}^2 \\
 & [x:y:z] & \mapsto & [y(x^2-y^2)\, : \,x(x^2-y^2)+zx^2 \, : \, y^3+z^3]
\end{array}
\]
First, it is easy to check that $\pi$ is a well-defined degree $9$ morphism, and the preimages of the singular point are given by: 
$\pi^{-1}([0:0:1])= \{[1:1:0], [1:-1:0], [0:0:1], [1:0:-1]\}$. 
As for the pull-back of the nodal curve, it can be decomposed as
\[\pi^{*}C=H+E\]
where $H=(z=0)$ and $E$ is the reduced divisor defined by $f(x,y,z)=0$ where
 \[f(x,y,z)=z^4x^4+z^2(x^2-y^2)^3+zx^4y^3+2x^3y^3(x^2-y^2)\]
 
 \vspace{3mm}
 How can one find such a morphism? One starts with the normalization of the affine cubic $\mathbb A^1 \to \mathbb A^2, x\mapsto (x^2-1, x(x^2-1))$ that we can compactify as the normalization of $C$: 
 \[
 \begin{array}{ccc}
 \P^1 &\longrightarrow & \mathbb{P}^2 \\
 
[x:y] & \mapsto & [y(x^2-y^2)\, : \, x(x^2-y^2) \, : \, y^3]
\end{array}
\]
Then, one tries to extends this morphism to a finite endomorphism of $\P^2$. So, this amounts to adding a polynomial divisible by $z$ to each component of $\pi_0$ and making sure that this gives a well-defined morphism. By construction, the pull-back of $C$ will contain $\P^1 \subset \P^2$. \\

We now aim to prove that $\pi$ satisfies the two conditions stated in the previous bullet points. We choose to work near the point $P=[1:1:0]$. Therefore, one can use the chart $x=1$ at the source, and the chart $z=1$ at the target. \\

$\bullet$ Let us show that $\pi$ is a local biholomorphism near $P$. In the above chosen coordinates, $\pi$ can be written as:
\[(y,z) \mapsto \left( \frac{y(1-y^2)}{y^3+z^3}, \frac{1-y^2+z}{y^3+z^3}\right) \]
so that 
\[ d\pi =\frac{1}{(y^3+z^3)^2} \begin{pmatrix}
(1-3y^3)(y^3+z^3)-3y^3(1-y^2) & -2y(y^3+z^3)-3y^2(1-y^2+z) \\
-3z^2y(1-y^2) & (y^3+z^3)-3z^2(1-y^2+z)\\
\end{pmatrix}\]
hence
\[ d\pi(1,0) = \begin{pmatrix}
-2 & -2 \\
0 & 1
\end{pmatrix}\]
 which is invertible, so that the claim is proved. \\
 
 $\bullet$ The global equation of $\pi^*C$ is $zf(x,y,z)=0$. We are going to show that $E=(f=0)$ is smooth at $P$ and that it meets $H=(z=0)$ transversely at this point. This will show that $P$ belongs to the snc locus of $H+E$. 
 
\noindent 
So we set $g(y,z):=f(1,y,z)$, and we compute: 
 \[\nabla g(y,z) = \begin{pmatrix}
 -6z^2y(1-y^2)^2+3zy^2+6y^2-10y^4\\
 4z^3+2(1-y^2)^3z+y^3
 \end{pmatrix}\]
which, at $(y,z)=(1,0)$ becomes: 
  \[\nabla g(1,0) = \begin{pmatrix}
-4\\
1
 \end{pmatrix}\]
Therefore, $E$ is smooth at $P$; moreover, $\nabla z = (0,1)$ is not collinear with $\nabla g(1,0)$. Thus, $H$ and $E$ meet transversely at this point, and the holomorphic implicit function theorem guarantees that $H+E$ has normal crossings at $P$. 
\backmatter

\bibliographystyle{smfalpha}
\bibliography{biblio}

\end{document}